\documentclass{article}
\usepackage{amsmath,amsthm,amssymb,amscd, url}

\newcommand{\cP}{\mathcal{P}}

\newcommand{\bbR}{\mathbb{R}}
\newcommand{\Fin}{\mathrm{Fin}}
\newcommand{\bbC}{\mathbb{C}}
\newcommand{\bzero}{{\bf 0}}

\newcommand{\ZF}{\mathsf{ZF}}
\newcommand{\AC}{\mathsf{AC}}

\newcommand{\bbQ}{\mathbb{Q}}
\newcommand{\bbF}{\mathbb{F}}
\newcommand{\CCR}{\mathsf{CC}_{\mathbb{R}}}

\newtheorem{theorem}{Theorem}[section]
\newtheorem{lemma}[theorem]{Lemma}

\newtheorem{corollary}[theorem]{Corollary}

\theoremstyle{definition}

\begin{document}

\title{Discontinuous homomorphisms, selectors and automorphisms of the complex field (in $\ZF$)
\footnote{2000 AMS subject classification 03E25; 12D99, 54H11. Keywords : Axiom of Choice, discontinuous homomorphisms
equivalence relations, automorphisms of the complex field.}}

\author{
Paul Larson
\thanks{Partially supported by NSF grant DMS-1201494.}
\\
Miami University\\
\and
Jind{\v r}ich Zapletal
\thanks{Partially supported by NSF grant DMS-1161078. 
The authors thank Christian Rosendal and Simon Thomas for helpful comments on this paper.}
\\University of Florida
}





\maketitle

\begin{abstract}
 We show, in Zermelo-Fraenkel set theory without the Axiom of Choice, that the existence of a discontinuous homomorphism between separable Banach spaces induces a selector for the Vitali equivalence relation $\bbR/\bbQ$. In conjunction with a result of Di Prisco and Todorcevic, this shows that a nonprincipal ultrafilter on the integers is not sufficient to construct a discontinuous automorphism of the complex field, confirming a conjecture of Simon Thomas. This is an improved version of \cite{LZpams}, which used a weak version of the Axiom of Choice for the same result.
\end{abstract}

Assuming the Zermelo-Fraenkel axioms for set theory ($\ZF$), the Axiom of Choice ($\AC$) implies that every vector space has a basis (in fact the two statements are equivalent over $\ZF$ \cite{Blass}).
The existence of a basis for the vector space $\bbR$ over the field of scalars $\bbQ$ in turn implies, in $\ZF$, the existence of a selector for the Vitali equivalence relation $\bbR/\bbQ$ (the equivalence relation on $\bbR$ defined by the formula $x - y \in \bbQ$) and the existence of a discontinuous homomorphism from the group $(\bbR, +)$ to itself (see \cite{Khara1, Khara2, Heil}, for instance). We show that the existence of a discontinuous homomorphism from $(\bbR, +)$ to itself implies the existence of a selector for $\bbR/\bbQ$. We do this without using the axiom $\CCR$, which asserts the existence of Choice function for each countable set of subsets of $\bbR$, which we did use to prove the same result in \cite{LZpams}. Our result applies to the additive group of any separable Banach space in place of $(\bbR, +)$.


A \emph{selector} for an equivalence relation $E$ on a set $X$ is a subset of $X$ meeting each $E$-equivalence class in exactly one point. The classical construction of a nonmeasurable Vitali set begins by using $\AC$ to find a selector for $\bbR/\bbQ$. Instead of $\bbR/\bbQ$ however we will work with the equivalence relation $E_{0}$ of mod-finite equivalence for subsets of $\omega$; our introduction of $\bbR/\bbQ$ is only for the expository benefit of readers who are less familiar with $E_{0}$. The equivalence relations $\bbR/\bbQ$ and $\cP(\omega)/E_{0}$ are both hyperfinite and nonsmooth, so Borel bi-embeddable (see \cite{DJK}), which implies among other things that the existence of a selector for either of these equivalence relations implies the existence of one for the other.

The result in this paper confirms a conjecture of Simon Thomas saying that the existence of a nonprincipal ultrafilter on the integers is consistent with the nonexistence of a discontinuous automorphism of the complex field. We briefly give some background information connecting our result to his conjecture. Let $P$ be the set of primes, and for each $p \in P$ let $\bar{\bbF}_{p}$ be the algebraic closure of the field $\bbF_{p}$ of size $p$. Given a nonprincipal ultrafilter $U$ on $P$, the $U$-ultraproduct $\prod_{U}\bar{\bbF}_{p}$ is an algebraically closed field of characteristic $0$ and cardinality $2^{\aleph_{0}}$. It follows that if $\AC$ holds (or just if there is a wellordering of $\cP(\omega)$) this ultraproduct is isomorphic to the complex field $(\bbC, +, \cdot)$ (see \cite{CK}, for instance). Even without $\AC$, this ultraproduct has $2^{\aleph_{0}}$ many automorphisms induced by the powers of the Frobenius automorphisms of the fields $\bar{\bbF}_{p}$ (see \cite{Garrett, LidlNied}).

Di Prisco and Todorcevic proved in \cite{DiPT} that a certain strong Ramsey principle for countable products of finite sets holds in Solovay's model $L(\bbR)$ from \cite{Solovay}.
This principle has implications for forcing extensions of $L(\bbR)$ via the partial order $\cP(\omega)/\Fin$ (such an extension has the form $L(\bbR)[U]$, where $U$ is a nonprincipal ultrafilter on $\omega$). For instance \cite{DiPT}, it implies that there is no $E_{0}$-selector in this model.
Thomas observed that this Ramsey principle also precludes the existence of an injection from $\prod_{U}\bar{\bbF}_{p}$ into $\bbC$ in $L(\bbR)[U]$.
He then conjectured that there are no discontinuous automorphisms of $(\bbC, + , \cdot)$ in this model, i.e., that the only automorphisms are
the identity function and complex conjugation. Our result confirms this conjecture, as the restriction of such an automorphism to $(\bbC, +)$ would be a discontinuous homomorphism. We state this formally in Theorem \ref{maincor} and Corollary \ref{autocor} below. We note that $\CCR$ holds in $L(\bbR)[U]$, as it is an inner model of a model $\AC$ with the same set of real numbers.

Let us say that an abelian topological group $(G, +)$ is \emph{suitable} if there is an invariant metric $d$ inducing the topology on $G$ such that
\begin{itemize}
\item $G$ is complete with respect to $d$;
\item letting $\bzero$ be the identity element of $G$, $d(\bzero, n\cdot x) = n \cdot d(\bzero, x)$ holds for all $x \in G$ and $n \in \omega$ (where $n \cdot x$ denotes the result of adding $x$ to itself $n$ times).
\end{itemize}
The additive group of a Banach space is suitable, under the metric given by the norm. Moreover, Theorem 1.2 of \cite{polymath} shows that a group is suitable if and only if it is isomorphic to closed subset of real Banach space under its addition operation. Note that the second condition above implies that a bounded metric cannot witness suitability. When working with a fixed suitable group $(G, +)$ and a witnessing metric $d_{G}$, we will write $\bzero_{G}$ for the identity element of $G$, $|x|_{G}$ for $d_{G}(\bzero_{G}, x)$ and $B_{G}(x, \epsilon)$ for $\{ y \in G : d(x,y) < \epsilon\}$.

\begin{lemma}[$\ZF$]\label{mainthrm}
  Suppose that $(G, +)$ and $(K, +)$ are suitable topological groups, and that $h \colon (G, +) \to (K, +)$ is a homomorphism.
   If there exists a convergent sequence $\langle x_{i} : i \in \omega \rangle$ in $G$ such that $\langle h(x_{i}) : i \in \omega \rangle$
   does not converge to $h(\lim_{n \in \omega} x_{i})$, then there is a selector for $E_{0}$.
\end{lemma}

\begin{proof}
  Let $d_{G}$ and $d_{K}$ be metrics witnessing the respective suitability of $(G, +)$ and $(K, +)$, and let $h$ and
   $\langle x_{i} : i \in \omega\rangle$ be as in the statement of the theorem. By the invariance of $d_{G}$ and $d_{K}$, it suffices to consider the case where
  $\langle x_{i} : i \in \omega \rangle$ converges to $\bzero_{G}$. Since $h(\bzero_{G}) = \bzero_{K}$, we have that $\langle h(x_{i}) :  \in \omega \rangle$ does not converge to $\bzero_{K}$, which means that for some $\epsilon > 0$ the set of $i \in \omega$ with $|h(x_{i})|_{K} \geq \epsilon$ is infinite.

  We may now find a sequence $\langle y_{i} : i \in \omega \rangle$ of elements of $G$ such that
  \begin{enumerate}
    \item for each $i \in \omega$ there exist $k \in \omega$ and $n \in \omega \setminus \{0\}$ such that $y_{i} = n \cdot x_{k}$;
    \item\label{condtwo} for all $i < j$ in $\omega$, $|y_{j}|_{G} < |y_{i}|_{G}/3$;
    \item\label{condthree} for all $i \in \omega$, $|h(y_{i})|_{K} > i + \sum_{j < i} |h(y_{j})|_{K}$.
  \end{enumerate}
  To see this, let $y_{0}$ be any element of $\{ x_{i} : i \in \omega\} \setminus \{\bzero_{G}\}$. Given $j \in \omega$ and $\{ y_{i} : i \leq j\}$,
  let $n \in \omega \setminus \{0\}$ be such that \[n\cdot \epsilon > (j+1) + \sum_{i < j + 1} |h(y_{i})|_{K}.\] There exists then a $k \in \omega$ such that $|x_{k}|_{G} < |y_{i}|_{G}/3n$ for all $i \leq j$ and such that $|h(x_{k})|_{K} \geq \epsilon$. Then $y_{j+1} = n\cdot x_{k}$ is as desired.

Condition (\ref{condtwo}) on $\langle y_{i} : i \in \omega \rangle$ implies that each value $|y_{i}|_{G}$ is more than $\sum \{ |y_{j}|_{G} : j > i\}$.
This in turn, along with the completeness of $G$, implies that $\sum_{i \in A} y_{i}$ converges for each $A \subseteq \omega$.
Let $Y = \{ y_{i}: i \in \omega \}$ and let $Y^{+}$ be the set of elements of $G$ which are sums of (finite or infinite) subsets of $Y$.
By condition (\ref{condtwo}) on $Y$, each $y \in Y^{+}$ is equal to $\sum \{ y_{i} : i  \in S_{y}\}$ for a unique subset $S_{y}$ of $\omega$.
Let $F$ be the equivalence relation on $Y^{+}$ where $y_{0} F y_{1}$  if and only if $S_{y_{0}}$ and $S_{y_{1}}$ have finite symmetric difference (i.e., $S_{y_{0}} E_{0} S_{y_{1}}$).
By condition (\ref{condthree}) on $\langle y_{i} : i < \omega \rangle$, if $y F y'$ and $i$ is the maximum point of disagreement between $S_{y}$ and $S_{y'}$, then
$d_{K}(h(y), h(y')) > i$. It follows that the $h$-preimage of each set of the form $B_{K}(\bzero_{K}, M)$ (for $M \in \bbR^{+}$) intersects each $F$-equivalence class in only finitely many points (since if $2M \leq i$, then for every $y$ in this intersection the set $S_{y} \setminus i$ is the same).
It follows from this (and the fact that there is a Borel linear order $<$ on $Y^{+}$ induced by the natural lexicographic order on $\cP(\omega)$) that there is an $F$-selector : for each $F$-equivalence class, let $M \in \mathbb{Z}^{+}$ be minimal so that the $h$-preimage of $B_{K}(\bzero_{K}, M)$ intersects the class, and then pick the $<$-least element of this intersection. Since $Y^{+}/F$ is isomorphic to $\cP(\omega)/E_{0}$ via the map $y \mapsto S_{y}$, there is then an $E_{0}$-selector.
\end{proof}

 Theorems \ref{maincor} and \ref{maincorb} are each applications of Lemma \ref{mainthrm}. A \emph{choice function} for a set $A$ is a function $c$ with domain $A \setminus \{0\}$ such that $c(a) \in a$ for all $a \in A \setminus \{0\}$.  If $D$ is dense subset of $G$, a choice function for the powerset of $D$ can be used to find convergent sequences from $D$. Recall that choice functions exist for the powerset of any wellorderable set, and therefore the powerset of any countable set, so the following theorem includes the case where $G$ is separable.

\begin{theorem}[$\ZF$]\label{maincor} Suppose that $(G, +)$ is a suitable group, $D$ is a dense subset of $G$, and that there exists a choice function for the powerset of $D$. If there is a discontinuous homomorphism of $(G, +)$ to itself, then there is a selector for $E_{0}$.
\end{theorem}

\begin{proof}
  Let $h$ be a discontinuous homomorphism from a suitable group $(G, +)$ to itself.
   Let $D$ be a wellorderable dense subset of $G$. Since $h$ is discontinuous, there exist a $x \in G$ and a sequence $\langle x_{n} : n \in \omega \rangle$ in $D$ such that $\langle x_{n} : n \in \omega \rangle$ converges to $x$ but $\langle h(x_{n}) : n \in \omega \rangle$ does not converge to $h(x)$.
  Now we may apply Theorem \ref{mainthrm}.
\end{proof}

Rephrasing in terms of Banach spaces gives the following.

\begin{corollary}[$\ZF$]\label{maincor2} If there is a discontinuous homomorphism between separable Banach spaces then there is a selector for $E_{0}$.
\end{corollary}

Lemma \ref{mainthrm} and Theorem \ref{maincor} do not require the Axiom of Choice, but in general it may require some form of Choice to find a sequence $\langle x_{i} : i < \omega \rangle$ as in the statement of Lemma \ref{mainthrm}, given a discontinuous homomorphism on a suitable group.
Theorem \ref{maincorb} applies to the case of (possibly nonseparable) groups of cardinality continuum. 

\begin{theorem}[$\ZF + \CCR$]\label{maincorb} If there is a discontinuous homomorphism between suitable groups of cardinality $2^{\aleph_{0}}$ then there is a selector for $E_{0}$.
\end{theorem}

\begin{proof}
  Let $(G, +)$ and $(K, +)$ be suitable groups of cardinality $2^{\aleph_{0}}$, and let $h$ be a discontinuous homomorphism from
  $(G, +)$ to $(K, +)$. Let $d_{G}$ and $d_{K}$ be metrics on $G$ and $K$ witnessing suitability.
  Since $h$ is discontinuous, and $d_{G}$ and $d_{K}$ are invariant, there exists an $\epsilon > 0$ such that for each $\delta > 0$ there exists an $x \in B_{G}(\bzero_{G}, \delta)$ with $h(x) \not\in B_{K}(\bzero_{K}, \epsilon)$. For each $i \in \omega$, let $X_{i}$ be the set of $x \in B_{G}(\bzero_{G}, 1/(i + 1))$ such that $h(x) \not\in B_{K}(\bzero_{K}, \epsilon)$. Then each $X_{i}$ is nonempty, and by $\CCR$ there is a sequence $\langle x_{i} : i \in \omega \rangle$ with each $x_{i}$ in the corresponding $X_{i}$. Now we may apply Lemma \ref{mainthrm}.
\end{proof}

Combined with the results of Di Prisco and Todorcevic cited above, we have the following corollary, which says that the assumption of the existence of a nonprincipal ultrafilter on the integers is not sufficient to define a third automorphism of the complex field.
The strongly inaccessible cardinal in the hypothesis (which we conjecture to be unnecessary) comes from the construction of the model $L(\bbR)$ in \cite{Solovay}.

\begin{corollary}\label{autocor}
  If the theory $\ZF$ is consistent with the existence of a strongly inaccessible cardinal, then it is also consistent with the conjunction of the following three statements:
  \begin{itemize}
  \item $\CCR$ holds;
  \item there is a nonprincipal ultrafilter on $\omega$;
  \item there are exactly two automorphisms of the complex field.
  \end{itemize}
\end{corollary}

This paper is part of the project outlined in \cite{LZ, LZ2}, which studies fragments of the Axiom of Choice holding in certain generic extensions of models of the Axiom of Determinacy. The following are shown in \cite{LZ2}, relative to the consistency of the existence of a strongly inaccessible cardinal.

\begin{enumerate}
\item The existence of an $E_{0}$ selector does not imply the existence of a discontinuous homomorphism on
$(\bbR, +)$.
\item If we drop the second condition from the definition of suitability, Theorem \ref{maincor} no longer holds. In particular, letting $(G, +)$ be the group induced by addition modulo $1$ on the interval $[0, 1)$, the existence of a discontinuous homomorphism of $(G,+)$ does not imply the existence of an $E_{0}$-selector. 
\end{enumerate}


Our proof of Theorem \ref{mainthrm} was discovered by adapting arguments from \cite{LZ}, with additional inspiration from \cite{Kestelman}.

We end with some related questions. The intended context for each question is the theory $\ZF + \CCR$, although the versions for other forms of $\AC$ may be interesting.
\begin{enumerate}

\item\label{q2} Does the existence of a discontinuous homomorphism on
$(\bbR, +)$ imply the existence of a Hamel basis for $\bbR$ over $\bbQ$?
\item Does the existence of a discontinuous homomorphism of $(\bbR, +)$ imply the existence of a discontinuous automorphism of $(\bbC, +, \cdot)$?
\end{enumerate}

We thank Paul McKenney for reminding us of Question (\ref{q2}).



\bibliographystyle{amsplain}

\end{document}